\documentclass[11pt, oneside, a4]{article}   	
\usepackage{amsmath, amsthm, amsfonts, amssymb}

\newcommand\at{{\tilde a}}
\newcommand\bt{{\tilde b}}

\newcommand\Ex{{\mathbf E}}

\newcommand\chitilde{\tilde{\chi}}
\newcommand\Normal{{\mathcal N}}

\newcommand\cF{{\mathcal F}}
\newcommand\cM{{\mathcal M}}

\newcommand\cP{{\mathcal P}}
\newcommand\cQ{{\mathcal Q}}

\newcommand\N{{\mathbb N}}

\newcommand\R{{\mathbb R}}

\newcommand\dto{\overset{d}{\to }}

\newcommand\one{{\bf 1}}

\newcommand\bra[1]{\langle #1 \rangle}

\DeclareMathOperator{\Gam}{Gamma}

\DeclareMathOperator{\Tr}{Tr}

\DeclareMathOperator{\Var}{Var}
\DeclareMathOperator{\Cov}{Cov}

\newtheorem{theorem}{Theorem}[section]
\newtheorem{corollary}[theorem]{Corollary}
\newtheorem{lemma}[theorem]{Lemma}

\theoremstyle{definition}

\theoremstyle{remark}

\long\def\symbolfootnote[#1]#2{\begingroup%
\def\thefootnote{\fnsymbol{footnote}}\footnote[#1]{#2}\endgroup}

\title{Global spectrum fluctuations for Gaussian beta ensembles: a martingale approach}
\author{Khanh Duy Trinh
 \footnote{Institute of Mathematics for Industry, Kyushu University, Japan. Email: trinh@imi.kyushu-u.ac.jp} 
}

\begin{document}
\maketitle
\begin{abstract}
The paper describes the global limiting behavior of Gaussian beta ensembles where the parameter $\beta$ is allowed to vary with the matrix size $n$. In particular, we show that as $n \to \infty$ with $n\beta \to \infty$, the empirical distribution converges weakly to the semicircle distribution, almost surely. The Gaussian fluctuation around the limit is then derived by a martingale approach.   

\medskip

		
\medskip
	
\end{abstract}

\section{Introduction}
Gaussian beta ensembles, as a generalization of Gaussian Othogonal/Unitary/Symplectic Ensembles (GOE, GUE and GSE for short), were originally defined as the ensembles of real particles with the following joint density function
\begin{equation}\label{GE-denstiy}
	p_{n, \beta}(\lambda_1, \dots, \lambda_n) = \frac{1}{Z_{n, \beta}}  |\Delta(\lambda)|^\beta  e^{-\frac{n \beta} 4(\lambda_1^2 + \cdots + \lambda_n^2)}, \quad (\beta > 0).
\end{equation}
Here $\Delta(\lambda) = \prod_{i<j}(\lambda_j - \lambda_i)$ denotes the Vandermonde determinant and $Z_{n, \beta}$ is a normalizing constant. Gaussian beta ensembles belong to a class of general beta ensembles which is related to many fields of mathematics and physics such as random matrix theory, spectral theory, representation theory and quantum mechanics. Based on the explicit joint density, the semicircle law and a Gaussian fluctuation around the limit were established by a method which is applicable for a large class of beta ensembles \cite{Johansson98}. Here the semicircle law means that the empirical distribution $L_n = n^{-1}(\delta_{\lambda_1} + \cdots + \delta_{\lambda_n})$, the probability measure which puts equal mass on each eigenvalue, converges weakly to the semicircle distribution, almost surely. More precisely, it means that for any bounded continuous function $f$, as $n \to \infty$,
\begin{equation}\label{LLN}
	\bra{L_n, f} =\frac{1}{n}\Big(f(\lambda_1) + \cdots + f(\lambda_n) \Big) \to \int_{-2}^2 f(x) \frac{\sqrt{4- x^2}}{2\pi} dx =\bra{sc, f} \text{ almost surely,}
\end{equation}
with $sc$ denoting the semicircle distribution, the probability measure on $[-2,2]$ with density $(2\pi)^{-1} \sqrt{4 - x^2}$. Moreover, when the test function $f$ is smooth enough, the following central limit theorem holds
\begin{equation}\label{CLT}
	f(\lambda_1) + \cdots + f(\lambda_n)  - m_{n,f} \dto \Normal(0, \frac{\sigma_f^2}{\beta}),
\end{equation}
where $m_{n, f}$ is non random and $\sigma_f^2$ is a constant which does not depend on $\beta$ and can be expressed as a quadratic functional of $f$, the notation``$\dto$'' denotes the convergence in distribution and $\Normal(\mu, \sigma^2)$ denotes a Gaussian distribution with mean $\mu$ and variance $\sigma^2$.

Gaussian beta ensembles are now realized as eigenvalues of symmetric tridiagonal matrices, called Jacobi matrices, with independent entries distributed according to a certain distribution. The model, which is based on  tridiagonalizing the GOE or GUE, was introduced in  \cite{DE02}. Let $T_{n,\beta}$ be a symmetric tridiagonal matrix with independent (up to the symmetric constraint) entries distributed as
\[
	T_{n, \beta} = \frac{\sqrt{2}}{\sqrt{n\beta}} \begin{pmatrix}
		\Normal(0,1)		&\chitilde_{(n - 1)\beta}	\\
		\chitilde_{(n - 1)\beta}	&\Normal(0, 1)		&\chitilde_{(n - 2)\beta}		\\
					
												&\ddots		&\ddots		&\ddots \\
						
						&&					\chitilde_{\beta}	&\Normal(0, 1)
	\end{pmatrix},
\] 
where $\chitilde_k$ denotes the distribution of the  square root of the gamma distribution with parameters $(k/2,1)$. Then the eigenvalues of $T_{n,\beta}$ are distributed according to Gaussian beta ensembles.

Spectrum properties of Gaussian beta ensembles have been intensively studied. Let us mention here some results for fixed $\beta$ on three main regimes (global, local and edge regimes). Typical results in the global regime are the law of large numbers~\eqref{LLN} and the central limit theorem~\eqref{CLT}. Such results can be proved in several ways \cite{DE06, Duy2016, Johansson98}. The local regime refers to the study of limiting behavior of eigenvalues near a fixed point $E\in (-2,2)$. Based on the tridiagonal matrix, the local statistics or bulk statistics is shown to converge to the Sine-$\beta$ point process \cite{Valko-Virag-2009}. In the edge regime, the largest eigenvalue, under suitable scaling, converges to the Tracy-Widom-$\beta$ distribution \cite{Ramirez-Rider-Virag-2011}.

The aim of this paper is to study the global regime of Gaussian beta ensembles when $\beta$ varies with $n$. It turns out that the global limiting behavior is governed by  the coupling parameter $n\beta$. Indeed, on the one hand, the semicircle law holds as long as $n\to \infty$ with $n\beta \to \infty$, which is the same as the case $\beta$ being fixed. This result can be viewed as a direct consequence of the large deviation principle for the empirical distribution \cite[Theorem~2.4]{Peche15}. On the other hand, when $n\beta$ stays bounded, the limiting behavior depends on the limit of the sequence $\{n\beta\}$. Specifically, as $n\to \infty$ with $n \beta \to 2\alpha \in (0, \infty)$, the empirical distribution converges to the limiting measure depending on $\alpha$, called the probability measure of associated Hermite polynomials \cite{Allez12, Peche15, Duy-Nakano-2016, DS15}.

 A natural question, of course, is about the fluctuation around the limit. Our main result gives the answer.
\begin{theorem}\label{thm:main}
\begin{itemize}
	\item[\rm(i)]
	As $n \to \infty$ with $n\beta \to \infty$, the empirical distribution $L_n$ converges weakly to the semicircle distribution, almost surely. 
	
	\item[\rm(ii)] Let $f$ be a differentiable function with continuous derivative of polynomial growth. Then as $n \to \infty$ with $n\beta \to \infty$, 
	\[
		 n\sqrt{\beta}(\bra{L_n, f} - \Ex[\bra{L_n, f}]) = \sqrt{\beta}\Big( \sum_{j = 1}^n f(\lambda_j) - \Ex \Big[ \sum_{j = 1}^n (f(\lambda_j)\Big]  \Big)\to \Normal(0, \sigma_f^2),
	\]
	where
\begin{equation}\label{variance-f}
		\sigma_f^2 = \frac{1}{2\pi^2} \int_{-2}^2 \int_{-2}^2 \left( \frac{f(x) -f(y)}{x - y}\right)^2 \frac{4 - xy}{\sqrt{4 - x^2}\sqrt{4 - y^2}} dx dy. 
	\end{equation}
\end{itemize}
\end{theorem}

For fixed $\beta$, the central limit theorem can be derived by analyzing the joint density \cite{Johansson98} or by combinatorics arguments based on the tridiagonal matrix model \cite{DE06}. These methods may also work for the case where $\beta$ varies. However, instead of refining them, we will introduce an approach based on martingale theory. The approach is natural and straightforward because the entries of tridiagonal matrices are independent. Note that the martingale theory has been used to establish such type of central limit theorem for several models in random matrix theory (see \cite{Shcherbina-2015} and the references therein).

Let us now explain the key idea of this paper. Let $p$ be a polynomial of degree $m > 0$. Then the linear statistics with respect to $p$ can be expressed as 
\[
	S_n =n\bra{L_n, p}= \sum_{j = 1}^n p(\lambda_j) = \Tr (p(T_{n, \beta})) = \sum_{j = 1}^n p(T_{n, \beta})(j,j).
\]
Since $T_{n, \beta}$ is a tridiagonal matrix, $p(T_{n, \beta})(j,j)$ depends only on entries near the $(j,j)$ location. In particular, $p(T_{n, \beta})(i, i)$ and $p(T_{n, \beta})(j,j)$ are independent if $|i - j| \ge m$. Thus, it should be easy to calculate the conditional expectation and  a central limit theorem for such linear statistics could be derived from the martingale difference central limit theorem. To extend the central limit theorem to a certain class of differentiable functions, we use the Poincar\'e inequality which is derived from the joint density to bound the variance. Then the result follows by a standard argument.

For the sake of completeness, we mention here the global behavior of Gaussian beta ensembles in the regime that $n\to \infty$ with $n \beta \to 2\alpha \in (0, \infty)$. 
\begin{theorem}[cf.~{\cite{Duy-Nakano-2016}}]\label{thm:alpha-regime}
\begin{itemize}
	\item[\rm(i)]
	As $n \to \infty$ with $n\beta \to 2\alpha \in (0, \infty)$, the empirical distribution $L_n$ converges weakly to the measure $\nu_\alpha$, almost surely. Here the density of $\nu_\alpha$ is given by $\nu_\alpha(x) = \sqrt{\alpha}\bar \mu_\alpha(\sqrt{\alpha} x)$ with 
	\[
		\bar\mu_\alpha(x) = \frac{e^{-x^2/2}}{\sqrt{2\pi}}\frac{1}{|\hat f_\alpha (x)|^2}, \text{where } \hat f_\alpha (x) = \sqrt{\frac{\alpha}{\Gamma(\alpha)}} \int_0^\infty t^{\alpha - 1} e^{-\frac{t^2}{2} + i x t} dt.
	\]
	\item[\rm(ii)] Let $f$ be a differentiable function with continuous derivative of polynomial growth. Then as $n \to \infty$ with $n\beta \to 2\alpha$, 
	\[
		 \sqrt{\beta} \Big( \sum_{j = 1}^n f(\lambda_j) - \Ex \Big[ \sum_{j = 1}^n (f(\lambda_j)\Big]  \Big)\to \Normal(0, \hat \sigma_{f,\alpha}^2),
	\]
	where $\hat \sigma_{f,\alpha}^2$ is a constant.
\end{itemize}
\end{theorem}
It was shown in \cite{Peche15, Duy-Nakano-2016} that the convergence of the empirical measure to a limit holds in probability. We strengthen it to the almost sure convergence in Lemma~\ref{lem:as}. Note that the martingale approach has already been used to derive the above central limit theorem. However, the proof in \cite{Duy-Nakano-2016} relied on some estimates  which hold only in the case $n\beta$ being bounded. In this paper, we generalize such martingale approach and make it easy to apply to general tridiagonal matrices with independent entries. For example, a central limit for general tridiagonal matrix models in \cite{Popescu-09} can also be proved by the approach here.

The paper is organized as follows. The proof of the semicircle law in case $n\beta \to \infty$ is given in Section 2. Section 3 deals with general tridiagonal matrices whose entries are independent. We give sufficient conditions under which a central limit theorem holds for the linear statistics of a polynomial test function. As an application of the general theory, Theorem~\ref{thm:main}(ii) is derived in Section 4. We end up with some remarks in Section 5.

\section{The semicircle law}
In this section, we show the almost sure convergence of the empirical distribution of Gaussian beta ensembles to a limit as $n \to \infty$ with $n\beta \to 2\alpha \in (0, \infty]$.

Recall that 
\[
	T_{n, \beta} = \frac{\sqrt{2}}{\sqrt{n\beta}} \begin{pmatrix}
		\Normal(0,1)		&\chitilde_{(n - 1)\beta}	\\
		\chitilde_{(n - 1)\beta}	&\Normal(0, 1)		&\chitilde_{(n - 2)\beta}		\\
					
												&\ddots		&\ddots		&\ddots \\
						
						&&					\chitilde_{\beta}	&\Normal(0, 1)
	\end{pmatrix}.
\] 
Let $\{\at_i\}_{i = 1}^n$ be a sequence of i.i.d.~(independent identically distributed) random variables with standard Gaussian distribution and let $\{\bt_i\}_{i = 1}^{n - 1}$ be a sequence of independent random variables which is also independent of $\{\at_i\}$ such that $\bt_i^2$ has the gamma distribution with parameters $({(n-i)\beta}/{2},1)$. 
Then the diagonal $\{a_i\}_{i = 1}^n$ and the sub-diagonal $\{b_i\}_{i = 1}^{n - 1}$ of $T_{n, \beta}$ are expressed as 
\[
	a_i =  \frac{\sqrt{2}}{\sqrt{n\beta}} \at_i, \quad b_i =  \frac{\sqrt{2}}{\sqrt{n\beta}} \bt_{i}.
\]

The following properties of the gamma distribution are useful for the next discussion.
\begin{lemma}
\begin{itemize}
	\item[\rm(i)] Let $X_n$ be a sequence of gamma distributed random variables with parameters $(\alpha_n,1)$. Assume that $\alpha_n \to \infty$ as $n \to \infty$. Then 
\[
	\frac{\sqrt{X_n}}{\sqrt{\alpha_n}} \to 1 \text{ in probability and in $L^q$ for any $q \in [1,\infty)$.}
\]
\item[\rm(ii)] 	Let $X$ be distributed as gamma distribution with parameters $(\alpha,1)$. Then for any nonnegative interger $k$,
\[
	\Ex[X^k] =\frac{\Gamma(\alpha + k)}{\Gamma(\alpha)} = \alpha (\alpha + 1) \cdots (\alpha + k -1 ).
\]
\end{itemize}

\end{lemma}

Let $\mu_n$ be the spectral measure of $T_{n, \beta}$, the measure defined by
\[	
	\langle \mu_n, x^k\rangle = T_{n,\beta}^k (1,1), k = 0,1,\dots.
\]
Then the spectral measure $\mu_n$ is supported on the eigenvalues with weights $\{q_j^2 = |v_j(1)|^2\}$, that is,
\[
	\mu_n = \sum_{j = 1}^n |v_j(1)|^2 \delta_{\lambda_j} = \sum_{j = 1}^n q_j^2 \delta_{\lambda_j},
\] 
where $v_1, \dots, v_n$ are the corresponding normalized eigenvectors of $T_{n, \beta}$. 
It is known that the weights $\{q_j^2\}$ are independent of the eigenvalues and have the same distribution with the following random vector
\[
	\left(\frac{\chitilde_{\beta,1}^{2}}{\chitilde_{\beta,1}^{2} + \cdots + \chitilde_{\beta,n}^{2}}, \dots, \frac{\chitilde_{\beta,n}^{2}}{\chitilde_{\beta,1}^{2} + \cdots + \chitilde_{\beta,n}^{2}}\right),
\]
where $\{\chitilde_{\beta, i}^2\}_{i = 1}^n$ is an i.i.d.~sequence whose common distribution is the gamma distribution with parameters  $(\frac \beta2, 1)$ (cf.~\cite[Theorem~2.12]{DE02}). That is to say, the weights $\{q_j^2\}$ are distributed according to the  Dirichlet distribution with parameters $(\beta/2, \dots, \beta/2)$.

When $n \to \infty$ with $n\beta \to \infty$, it follows directly from the asymptotic of the gamma distribution that the matrix $T_{n, \beta}$ converges entrywise to the free Jacobi matrix $J_{free}$
\[
T_{n,\beta}
	\to
	\begin{pmatrix}
	0	&1\\
	1	&0	&1\\
	&\ddots	&\ddots	&\ddots
	\end{pmatrix}	=:J_{free},
\]
namely, for any fixed $i$, as $n \to \infty$ with $n\beta \to \infty$,
\[
	T_{n, \beta}(i,i) \to 0; \quad T_{n,\beta}(i, i+1)\to 1 \text{ in probability and in $L^q$ for any $q \in [1,\infty)$}.
\]
Note that the spectral measure of the free Jacobi matrix, the probability measure $\mu$ satisfying $\bra{\mu, x^k} = J_{free}^k(1,1)$ for all $k = 0,1,\dots$, is nothing but the semicircle distribution. Let $p$ be a polynomial. Then $\bra{\mu_n, p} = p(T_{n,\beta})(1,1)$ is a multivariate polynomial of some top-left entries. Therefore, in this regime, the sequence $\{\bra{\mu_n, p}\}$ converges to $\bra{sc, p}$ in probability and in $L^q$ for any $q \in [1, \infty)$. Consequently, the spectral measure $\mu_n$ converges weakly to the semicircle distribution, in probability. A rigorous proof of the above argument can be found in \cite{Duy2016}.

In the regime where $n\beta \to 2\alpha \in (0, \infty)$, the matrix $T_{n, \beta}$ converges entrywise in distribution to the following infinite Jacobi matrices with independent entries
\[
	J_\alpha = \frac{1}{\sqrt{\alpha}} \begin{pmatrix}
		\Normal(0,1)	&\chitilde_{2\alpha}\\
		\chitilde_{2\alpha}	&\Normal(0,1)	&\chitilde_{2\alpha}\\
		&\ddots	&\ddots	&\ddots
	\end{pmatrix}.
\]
From which, we can deduce that for any polynomial $p$, the sequence $\Ex[\bra{\mu_n, p}]$ converges to a finite limit. The limit turns out to coincide with $\bra{\nu_\alpha, p}$, where $\nu_\alpha$ is the probability measure defined in Theorem~\ref{thm:alpha-regime} (cf.~\cite{DS15}). It is worth mentioning that the probability measure $\nu_\alpha$ is determined by moments.

For Gaussian beta ensembles, since the weights $\{q_j^2\}$ are independent of the eigenvalues, it follows that 
\[
	\Ex[\bra{\mu_n, p}] = \Ex\bigg[\sum_{j = 1}^n q_j^2 p(\lambda_j) \bigg] = \sum_{j = 1}^n \Ex[q_j^2] \Ex[p(\lambda_j)] = \Ex\bigg[\frac{1}{n}\sum_{j = 1}^n p(\lambda_j) \bigg] = \Ex[\bra{L_n, p}].
\]
Here we have used the fact that $\Ex[q_j^2] = 1/n$. Then we arrive at the following result for empirical distributions. 
\begin{lemma}\label{lem:ex}
	For any polynomial $p$, as $n \to \infty$ with $n\beta \to 2\alpha \in (0, \infty]$, 
	\[
		\Ex[\bra{L_n, p}] = \Ex[\bra{\mu_n, p}] \to \bra{\nu_\alpha, p}.
	\]
Here, for convenience, $\nu_\infty$ is used to denote the semicircle distribution. 
\end{lemma}

The convergence of expectation plays a crucial role in establishing the following strong law of large numbers.

\begin{theorem}\label{thm:LLN}
	As $n \to \infty$ with $n\beta \to 2\alpha \in (0,\infty]$, the empirical distribution of Gaussian beta ensembles converges weakly to $\nu_\alpha$, almost surely.
\end{theorem}

Let $\cP_r$ be the set of all closed paths of length $r$ starting at zero, $w = (i_0=0, i_1, \dots, i_r = 0)$, satisfying $|i_{l + 1} - i_l| \le 1, l = 0,1,\dots, r - 1$. For $j \in \N$ and $w \in \cP_r$, let $j + w$ denote the shifted path $(j + i_0, j+i_1, \dots, j + i_r)$, which is a closed path of length $r$ starting at $j$. A path $\hat w$ is called admissible if $\hat w = j + w \in j + \cP_r$ with $j + i_l \in \{1,2, \dots, n\}, l = 0,1,\dots, r$. Then for each $w \in \cP_r$, there are two indices $k_1=k_1(w)$ and $k_2=k_2(w)$ such that $j + w$ is admissible for all $k_1 \le j \le n - k_2$, provided that $n \ge k_1 + k_2$.

For an admissible path $w=(i_0, \dots, i_r)$, we define a random variable 
\[
	\xi_w = \prod_{l = 0}^{r - 1} T_{n, \beta}(i_l, i_{l + 1}) = \prod_{i} a_i^{\alpha_i} b_i^{2\gamma_i},
\]
where
\[
	\alpha_i = \alpha_i(w) = \#\{l : i_{l + 1} = i_l = i\}, \quad \gamma_i = \gamma_i(w) = \# \{j: i_l = i, i_{l + 1} = i + 1\}.
\]
Let $i_{\min} = \min_l{i_l}$ and $i_{\max} = \max_l {i_l}$. Since $w$ is a closed path in which each step can go at most $1$, it follows that $i_{\max} - i_{\min} \le r/2$. Thus given $i_0 = j$, $\xi_w$ depends at most on $\{a_i, b_i\}_{j - r/2, j + r/2}$.

The $r$th moment of the empirical distribution $L_n$ can be written as 
\[
	\bra{L_n, x^r} = \frac{1}{n}\Tr(T_{n, \beta}^r)=\frac{1}{n}\sum_{w \in \cP_r}  \sum_{j = k_1(w)}^{k_2(w)} \xi_{j + w}.
\]
Here $\Tr(A)$ denotes the trace of a matrix $A$.

Let $p = c_0 + c_1 x + \cdots + c_m x^m$ be a polynomial of degree $m>0$. Then
\[
	\langle L_n, p \rangle = \frac{1}{n}\sum_{j = 1}^n p(\lambda_j) =\frac{1}{n}\Tr(p(T_{n, \beta}))= \frac1n \sum_{j = 1}^n p(T_{n, \beta})(j,j).
\]
Set $z_j = p(T_{n, \beta})(j, j) - \Ex[p(T_{n, \beta})(j, j)]$. Observe that $z_i$ and $z_j$ are independent if $|i - j| \ge m$. Thus we separate the following sum into $m$ sums where each sum is a sum of independent random variables 
\[
	n(\langle L_n, p \rangle - \Ex[\langle L_n, p \rangle]) = \sum_{i = 1}^m \left (\sum_{l}z_{i + l m} \right) = \sum_{i = 1}^m Z_i.
\]

Note that
\[
	p(T_{n, \beta})(j,j) = \sum_{r = 0}^m c_r  \sum_{w \in  (j + \cP_r)^*} \xi_w,
\]
where $ (j + \cP_r)^*$ is the set of admissible paths of length $r$ starting at $j$, which is a subset of $j + \cP_r$. In particular, in the expression of $p(T_{n, \beta})(j,j)$, the number of terms is bounded by a constant $M$ which depends only on $m$. We write $M=M(m)$ for short. In addition,  for $n\beta \ge \varepsilon$, there is a constant $L=L(m, \varepsilon)$ such that 
\[
	\sup_{\substack{1\le i \le n \\ r \le 4m}}{\Ex[|a_i|^{r}]} \le L, \quad \sup_{\substack{1\le i \le n-1 \\ r \le 4m}} \Ex[b_i^{r}]  \le L.
\]
Therefore, there is a constant $C = C(p, \varepsilon)$ such that 
\[
	\Ex[z_j^2] \le C,  \quad \Ex[z_j^4] \le C,
\]
provided that $n\beta \ge \varepsilon$. It follows directly that 
\[
	n^2\Ex[(\langle L_n, p \rangle - \Ex[\langle L_n, p \rangle])^2] \le m \sum_{i = 1}^m \Ex[Z_i^2] = m \sum_{i = 1}^m \sum_{l}\Ex[z_{i+ lm}^2] \le mCn.
\]
Consequently, in the regime that $n \to \infty$ with $\liminf_{n\to \infty} n\beta > 0$, 
\begin{equation}\label{bound-for-variance}
	\Var[\langle L_n, p \rangle] \to 0. 
\end{equation}
From which, the convergence in probability of the empirical distribution to a limit follows.

In order to show the almost sure convergence, we will calculate the fourth moment and show that they are summable. We begin with the following estimate 
\[
	n^4\Ex[(\langle L_n, p \rangle - \Ex[\langle L_n, p \rangle])^4] \le m^3 \sum_{i = 1}^m \Ex[Z_i^4] .
\]
Since $Z_i$ is the sum of mean zero independent random variables, its fourth moment can be simplified as  
\[
	\Ex[Z_i^4] = \sum_{i_1, i_2, i_3, i_4} \Ex[z_{i_1}z_{i_2}z_{i_3}z_{i_4}] =6\sum_{i_1 < i_2} \Ex[z_{i_1}^2] \Ex[z_{i_2}^2] + \sum_{i_1}\Ex[z_{i_1}^4],
\]
which is bounded by $(3 C^2 n_i^2 + Cn_i)$, where $n_i$ is the number of terms in $Z_i$. Therefore, there is a constant $C'=C'(p, \epsilon)$ such that for $n\beta \ge \varepsilon$,
\begin{equation}\label{4thmoment}
	\Ex[(\langle L_n, p \rangle - \Ex[\langle L_n, p \rangle])^4] \le \frac{C'}{n^2}.
\end{equation}
The above estimate is enough to prove the following strong law of large numbers.
\begin{lemma}\label{lem:as}
	For any polynomial $p$, as $n \to \infty$ with $\liminf_{n \to \infty} n\beta > 0$,
	\[
		\langle L_n, p \rangle - \Ex[\langle L_n, p \rangle] \to 0 \text{ almost surely.}
	\]
\end{lemma}
\begin{proof}
	Let $\beta(n)$ be a sequence of positive real numbers with $\liminf_{n \to \infty} n \beta(n) >0$. Then the estimate~\eqref{4thmoment} implies that 
	\[
		\sum_{n = 1}^\infty \Ex[(\langle L_n, p \rangle - \Ex[\langle L_n, p \rangle])^4] < \infty.
	\]
It follows that 
\[
	\sum_{n = 1}^\infty (\langle L_n, p \rangle - \Ex[\langle L_n, p \rangle])^4 < \infty \text{ almost surely},
\]
and thus, 
\[
	\langle L_n, p \rangle - \Ex[\langle L_n, p \rangle] \to 0 \text{ almost surely.}
\]
The proof is complete.
\end{proof}

\begin{proof}[Proof of Theorem~\rm{\ref{thm:LLN}}]
Lemma~\ref{lem:ex} and Lemma~\ref{lem:as} imply that in the regime where $n\beta \to 2\alpha \in (0, \infty]$,
\[
	\bra{L_n, p} \to \bra{\nu_\alpha, p} \text{ almost surely},
\]
for any polynomial $p$. Since the probability measure $\nu_\alpha$ is determined by moments, the almost sure convergence of the sequence of empirical distributions $\{L_n\}$ follows by a standard argument (see Appendix B in \cite{Duy-Nakano-2016}, for example). The proof is complete.
\end{proof}

The rest of this section is devoted to study the limiting behavior of the sequence of the variances $\Var[\bra{L_n, p}]$. We will need it in establishing the central limit theorem. The following results show a further relation between the empirical distribution and the spectral measure. It is also an easy consequence of the aforementioned properties of the weights $\{q_j^2\}$. 
\begin{lemma}\label{lem:relation-spectral-measure}
For any test function $f$, it holds that
\begin{equation}\label{relation-of-variance}
	n\beta \Var[\bra{L_n, f}]  = (n\beta  + 2) \Ex[\bra{\mu_n, f}^2] - 2 \Ex[\bra{\mu_n, f^2}]  -  n\beta \Ex[\bra{\mu_n, f}]^2.
\end{equation} 
\end{lemma}
\begin{proof}
Since $\{q_j^2\}_{j = 1}^n$ has Dirichlet distribution with parameter $(\beta/2, \dots, \beta/2)$, we have 
\[
	\Ex[q_j^2] = \frac{1}{n}, \Ex[q_j^4] = \frac{\beta + 2}{n(n\beta + 2)}, \Ex[q_i^2 q_j^2] = \frac{\beta}{n(n\beta + 2)}, (1\le i \ne j \le n).
\]
In addition, the weights $\{q_j^2\}$ are independent of the eigenvalues $\{\lambda_j\}$. Then the rest of the proof follows by a direct calculation.
\end{proof}

Note that the $r$th moment of the spectral measure $\mu_n$ depends only on $\{a_i, b_i\}_{1\le i \le (r+1)/2}$. Then by using moments of the gamma distribution, we will deduce the following expressions.
\begin{lemma}\label{lem:single-moment}
\begin{itemize}
\item[\rm(i)] The expectation of the $r$th moment of the spectral measure $\mu_n$ can be  expressed as follows
	\begin{equation}
		\Ex[\bra{\mu_n, x^r}] = \begin{cases}
			0, &\text{if $r= 2q+1$},\\
			\sum_{k = 0}^{r/2} \frac{R_{r; k}(\beta)}{(n\beta)^k}, &\text{if $r= 2q$}.\\
		\end{cases}
	\end{equation}
Here $R_{r; k}(\beta)$ is a polynomial in $\beta$ of degree at most $k$.
\item[\rm (ii)] For the product of two moments, it holds that 
	\begin{equation}
		\Ex[\bra{\mu_n, x^r}\bra{\mu_n, x^s}] =
		\begin{cases}
			0, &\text{if $r+s$ is odd,}\\
			\sum_{k = 0}^{(r+s)/2} \frac{Q_{r,s; k}(\beta)}{(n\beta)^k}, &\text{if $r+s$ is even,}
		\end{cases}
	\end{equation}
where $Q_{r, s; k}(\beta)$ is a polynomial in $\beta$ of degree at most $k$.
\end{itemize}
\end{lemma}
\begin{proof}
A Motzkin path of length $r$ is a path in $\cP_{r}$ which never goes below zero. Let $\cM_r$ be the set of all Motzkin paths of length $r$. For convenience, we shift the indices of the matrix $T_{n,\beta}$ by $1$ for which indices start from $0$. Then it is clear that 
\begin{align*}
	\bra{\mu_n, x^r} &= T_n^r(0,0) = \sum_{w \in \cM_r} \xi_w, \\
	\bra{\mu_n, x^r} \bra{\mu_n, x^s} &= \sum_{w_1 \in \cM_r} \sum_{w_2 \in \cM_s} \xi_{w_1} \xi_{w_2},
\end{align*}
provided that $n \ge \max\{(r + 1)/2, (s+1)/2\}$. However, by setting $b_n = 0$, the above equalities hold and the following arguments work for any $n \ge 1$.
The case where $r$ is odd or $(r+s)$ is odd is trivial because any odd moment of a Gaussian distribution is zero. In the remaining, recall that 
\[
	b_i^2 \sim \frac{2}{n\beta} \Gam \left(\frac{(n-i)\beta}{2},1\right) = \frac{2}{x}\Gam \left(\frac{x}{2} - \frac{i\beta}{2},1\right) \quad (x := n\beta).
\]
Then regard moments of $b_i^{2\gamma}$ as polynomials in $x$, the desired results easily follow.
\end{proof}

%
%
%

\begin{lemma}
Let $p$ be a polynomial of degree $m$. Then 
	\begin{equation}
		\Var[\bra{L_n, p}] = \sum_{k=2}^{m + 1} \frac{ \beta \ell_{p; k}(\beta)}{(n\beta)^k},
	\end{equation}
where $\ell_{p; k}(\beta)$ is a polynomial in $\beta$ of degree at most $k-2$.
\end{lemma}

\begin{proof}
It follows from Lemma~\ref{lem:single-moment} and the relation \eqref{relation-of-variance}  that $\Var[\bra{L_n, p}]$ has the following form
	\[
		\Var[\bra{L_n, p}] = \sum_{k = 0}^{m + 1} \frac{P_{p; k} (\beta)}{(n\beta)^k},
	\]
with $P_{p; k}(\beta)$ being polynomial in $\beta$ of degree at most $k$. Then  the constant term should be zero because of the estimate \eqref{bound-for-variance}.

Next, observe that for fixed $n$, $\Var[\bra{L_n, p}] \to 0$ as $\beta \to \infty$. Thus for any $n$,
\[
	\sum_{k = 1}^{m + 1} c_{k} n^{-k} = \lim_{\beta \to \infty} \Var[\bra{L_n, p}] = 0,
\]
where $c_{k}$ denotes the coefficient of $\beta^k$ in $P_{p; k}$. Consequently, all coefficients $c_k$ are zero, which implies that the degree of $P_{p;k}$ is at most $(k - 1)$.

Finally, let $n \to \infty$ with $n\beta \to \alpha \in (0, \infty)$, it also follows from the estimate~\eqref{bound-for-variance} that 
\[
	\sum_{k = 1}^{m + 1} P_{p; k}(0) \alpha^{-k} = \lim_{n\beta \to \alpha} \Var[\bra{L_n, p}] = 0,
\]
which then implies that $P_{p; k}(0) = 0$ for all $k$. In particular, $P_{p; 1} \equiv 0$. Combining all the facts, we deduce that $P_{p; k}(\beta) = \beta \ell_{p; k}(\beta)$ with $\ell_{p; k}(\beta)$ being a polynomial of degree at most $(k-2)$.
 The proof is complete.
\end{proof}

\begin{corollary}\label{cor:variance}
Let $p$ be a polynomial of degree $m$. Then as $n \to \infty$,
	\[
		n^2 \beta \Var[\bra{L_n, p}] \to 
		\begin{cases}
			 \ell_{p; 2}=:\sigma_p^2 , &\text{if } n\beta \to \infty, \\
			 \sum_{k = 2}^{m + 1} \frac{\ell_{p; k}(0)}{(2\alpha)^{k - 2}}=:\sigma_{p, \alpha}^2, &\text{if } n\beta \to 2\alpha \in (0, \infty).
		\end{cases}
	\]
\end{corollary}


\section{General Jacobi matrices with independent entries}
Consider a sequence of Jacobi matrices
\[
	J_n = \begin{pmatrix}
	a_1^{(n)}		&b_1^{(n)}		\\
	b_1^{(n)}		&a_2^{(n)}		&b_2^{(n)}		\\
	&\ddots	&\ddots	&\ddots	\\
	&&b_{n - 1}^{(n)}	&a_n^{(n)}
	\end{pmatrix}, \quad (a_i^{(n)} \in \R, b_i^{(n)} > 0),
\]
where for each $n$, the entries $\{a_i^{(n)}\}_{i = 1}^n$ and $\{b_i^{(n)}\}_{i = 1}^{n - 1}$ are assumed to be independent. For simplicity, we omit the superscript ${}^{(n)}$ in formulae.

Let $p$ be a polynomial of degree $m$. Based on the martingale difference central limit theorem,
we will give some sufficient conditions under which a central limit theorem holds for the sequence 
\[
	S_n = n \bra{L_n, p} = \sum_{j = 1}^n p(J_n)(j, j).
\]
Define a filtration $\cF_k = \sigma(a_i, b_i : i=1, \dots, k), k= 1,\dots, n$ and $\cF_0 = \{\emptyset, \Omega\}$. Let 
\begin{align*}
	X_k &= \Ex[S_n | \cF_k], 0 \le k \le n, \\
	Y_k &= X_k - X_{k - 1}, 1 \le k \le n,\\
	\sigma_k^2 &= \Ex[Y_k^2 | \cF_{k - 1}], 1 \le k \le n.
\end{align*}
Then a central limit theorem holds under the following conditions (cf.~\cite[Theorem~35.12]{Billingsley-PnM}).
\begin{lemma} \label{lem:MTGCLT}
Assume that there exists a sequence $\{v_n\}$ of positive numbers such that 
\[
	v_n^2\sum_{k = 1}^n \sigma_{k}^2 \to \sigma^2 \text{ in probability as }n \to \infty,
\]
where $\sigma^2$ is a positive constant, and that for each $\varepsilon >0$,
\[
	\sum_{k = 1}^n \Ex[(v_nY_{k})^2 {\bf 1}_{\{|v_nY_{k}| \ge \varepsilon\}}] \to 0 \text{ as }n\to \infty.
\]
Then
\[
	v_n(S_n - \Ex[S_n]) = v_n \sum_{k = 1}^n Y_{k} \overset{d}{\to} \mathcal N(0, \sigma^2) \text{ as }n\to \infty.
\]
\end{lemma}

Let $\cQ_m(n)$ be the set of all admissible paths of length at most $m$. Then $S_n$ is a linear combination of $\{\xi_w\}_{w \in \cQ_m(n)}$.  For fixed $k$, let
\[
	\xi_w^{(-)} = \prod_{i < k}  a_i^{\alpha_i} b_i^{2\gamma_i}, \quad \xi_w^{(+)} = \prod_{i > k} a_i^{\alpha_i} b_i^{2\gamma_i}.
\]
Then  
\[
	\xi_w = \xi^{(-)}_w a_k^{\alpha_k} b_k^{2\gamma_k} \xi^{(+)}_w .
\]

\begin{lemma}\label{lem:conditional-ex}
\begin{itemize}
\item[\rm (i)] For any path $w$,
	\begin{equation}
		\Delta_k(\xi_w) := \Ex[\xi_w | \cF_k] - \Ex[\xi_w | \cF_{k - 1}] =  \xi_w^{(-)} \Big(a_k^{\alpha_k} b_k^{2\gamma_k} - \Ex[a_k^{\alpha_k} b_k^{2\gamma_k}] \Big) \Ex[\xi_w^{(+)}].	
	\end{equation}
In particular, $\Delta_k(\xi_w)  = 0$, if $\alpha_k = 0$ and $\gamma_k = 0$. 
	
\item[\rm (ii)]	For two paths $w$ and $\hat w$, 
\begin{align*}
	&\Ex[\Delta_k(\xi_w) \Delta_k(\xi_{\hat w}) | \cF_{k - 1}] \\
	&=  \xi_w^{(-)}\xi_{\hat w}^{(-)} \Ex\Big[\Big(a_k^{\alpha_k} b_k^{2\gamma_k} - \Ex[a_k^{\alpha_k} b_k^{2\gamma_k}] \Big)\Big(a_k^{\hat\alpha_k} b_k^{2\hat \gamma_k} - \Ex[a_k^{\hat \alpha_k} b_k^{2\hat \gamma_k}] \Big) \Big] \Ex[\xi_w^{(+)}\xi_{\hat w}^{(+)}].
\end{align*}
Here $\hat \alpha_k$ and $\hat \gamma_k$ denote the corresponding quantities of the path $\hat w$.
\end{itemize}
\end{lemma}

Recall that for an admissible path  $w=(i_0 , i_1, \dots, i_r )$ of length $r$, $i_{\max} - i_{\min} \le r/2$, where $i_{\min} = \min_l{i_l}$ and $i_{\max} = \max_l {i_l}$. In addition, Lemma~\ref{lem:conditional-ex}(i) implies that $i_{\min} \le k  \le i_{\max}$, if $\Delta_k(\xi_w) \ne 0$. Since $Y_k$ is a linear combination of $\Delta_k(\xi_w)$ over admissible paths of length at most $m$, it follows that $Y_k$, and hence $\sigma_k^2$ depends only on $\{a_i, b_i\}_{k-m/2 \le i \le k+ m/2}$. Note that  
$Y_{k + x}$ is just a shifted version of $Y_k$ by shifting the indices of $a$'s and $b$'s, if $m \le k < k + x \le n - m$. In particular, when $\{a_i\}$ and $\{b_i\}$ are i.i.d.~sequences, then $\{Y_k\}$ and $\{\sigma_k^2\}$ are stationary for ${m \le k \le n - m }$. In this case, the two conditions in Lemma~\ref{lem:MTGCLT} hold trivially with $v_n = n^{-1/2}$ under the assumption that all moments of $a$'s and $b$'s are finite. In non i.i.d.~case, by observing that $\sigma_k^2$ and $\sigma_{k'}^2$ are independent when $|k - k'| > m$, we give the following useful criterions.
\begin{lemma}\label{lem:general-CLT}
	Assume that there exists a sequence of positive integers $\{v_n\}$ such that
	\begin{itemize}
		\item[\rm(i)]
			$
				v_n^2 \sum_{k = 1}^n \Ex[Y_k^2] \to \sigma_p^2 \text{ as } n \to \infty,
			$
			where $\sigma_p^2 \ge 0$ is a constant; and that
		\item[\rm(ii)]
			$
				v_n^4\Ex[Y_k^4] \le \frac{C}{n^{1+\delta}} ,
			$
		for all $k = 1, \dots, n$, where $C > 0$ and $\delta > 0$ are constants not depending on $n$.
Then 
	\[
		v_n (S_n - \Ex[S_n])= v_n \sum_{k = 1}^n Y_{k}  \dto \Normal(0, \sigma_p^2).
	\]
		
	\end{itemize}
\end{lemma}
\begin{proof}
By conditional Jensen's inequality and the assumption (ii), it follows that
\[
	v_n^4|\Cov(\sigma_k^2, \sigma_{k'}^2)| \le v_n^4(\Ex[\sigma_k^4] \Ex[\sigma_{k'}^4])^{1/2} \le v_n^4(\Ex[Y_k^4] \Ex[Y_{k'}^4])^{1/2} \le \frac{C}{n^{1+\delta}}.
\]
In addition recall that $\sigma_k^2$ and $\sigma_{k'}^2$ are independent, if $|k - k'| > m$. Thus 
\[
	0\le \Var\bigg[v_n^2\sum_{k = 1}^n \sigma_k^2 \bigg] = v_n^4\sum_{k, k'} \Cov(\sigma_k^2, \sigma_{k'}^2) \le \frac{(m+1) C}{n^\delta} \to 0 \text{ as } n \to \infty.
\]
On the other hand, since $\Ex[\sigma_k^2] = \Ex[Y_k^2]$,  the assumption (i) implies that 
\[
	v_n^2 \sum_{k = 1}^n \Ex[\sigma_k^2] \to \sigma_p^2 \text{ as }n \to \infty.
\]
Therefore 
\[
	v_n^2 \sum_{k = 1}^n \sigma_k^2 \to \sigma_p^2 \text{ in probability as } n \to \infty,
\]
which shows the condition (i) in Lemma~\ref{lem:MTGCLT}. The condition (ii) in Lemma~\ref{lem:MTGCLT}   also holds because
\[
	\sum_{k = 1}^n \Ex[(v_nY_k)^2 \one_{|v_n Y_k| \ge \varepsilon}] \le \frac{1}{\varepsilon^2} \sum_{k = 1}^n \Ex[(v_nY_k)^4]  \le \frac{C}{n^\delta \varepsilon^2} \to 0.
\]
Thus, the desired central limit theorem follows. The proof is complete.
\end{proof}

Here is the main result in this section.
\begin{theorem}\label{thm:general-CLT}
Let $p$ be a polynomial of degree $m$. Let 
\[
	S_n = \Tr(p(J_n)) = \sum_{j = 1}^n p(J_n)(j,j).
\]
	Assume that there exists a sequence of positive integers $\{v_n\}$ such that
	\begin{itemize}
		\item[\rm(i)]
			$
				v_n^2 \Var[S_n] \to \sigma_p^2 \text{ as } n \to \infty,
			$
			where $\sigma_p^2\ge 0$ is a constant; and that
		\item[\rm(ii)]
			$
				v_n^4\Ex[\Delta_k(\xi_w)^4] \le \frac{C}{n^{1+\delta}} ,
			$
		for all admissible paths $w \in \cQ_{m}(n)$, for all $k = 1, \dots, n$, where $C > 0$ and $\delta > 0$ are constants not depending on $n$.
Then 
	\[
		v_n (S_n - \Ex[S_n]) \dto \Normal(0, \sigma_p^2).
	\]
	\end{itemize}
\end{theorem}
\begin{proof}
	It is a direct consequence of Lemma~\ref{lem:general-CLT}. Indeed, the assumption (i) just rephrases the condition (i) in that lemma. For given polynomial $p$ of degree $m$, note that  
	\[
		Y_k = \sum_{w} c_w \Delta_k(\xi_w),
	\]
with at most $M$ terms, and $|c_w| \le L$, where $M$ and $L$ are constants depending only on $p$. Then the condition (ii) in Lemma~\ref{lem:general-CLT} immediately follows from  the assumption (ii). The proof is complete.  
\end{proof}

\section{Gaussian beta ensembles}
\subsection{A central limit theorem for polynomial test functions}
By using a general theory developed in the previous section, we first derive a central limit theorem for linear statistics of polynomial test functions. Since the condition (i) in Theorem~\ref{thm:general-CLT} has been established in Corollary~\ref{cor:variance}, we only need to show the condition (ii). For this purpose, we will need the following property of the gamma distribution.
\begin{lemma}\label{lem:gamma-distribution}
	Let $X$ be distributed as gamma distribution with parameters $(\alpha, 1)$. Then for any positive integer $r$, the fourth central moment of $X^r$, 
$
		\Ex[ (X^r - \Ex[X^r])^4] ,
$
is a polynomial in $\alpha$ of degree at most $(4r - 2)$.
\end{lemma}
\begin{proof}
	It is clear that $\Ex[ (X^r - \Ex[X^r])^4] $ is a polynomial in $\alpha$ of degree at most $4r$. It will be an easy exercise to check that the coefficients of $\alpha^{4r}$ and $\alpha^{4r - 1}$ vanish. 
\end{proof}

\begin{lemma}
There is a constant $C = C(m, \varepsilon)$ such that  
\[
	\Ex[\Delta_k(\xi_w)^4] \le \frac{C}{(n\beta)^2}.
\]
for all paths $w$ of length at most $m$ and $n\beta \ge \varepsilon$.
\end{lemma}
\begin{proof}
	Given $m\in \N$ and $\varepsilon > 0$, there exists a constant $M = M(m, \varepsilon)
$ such that for all $i=1,\dots, n$, and all $r \le 4m$,
\[
	\Ex[|a_i|^r] \le M, \quad \Ex[|\at_i|^r] \le M, \quad \Ex[|b_i|^r] \le M.
\]

For a path $w$ of length at most $m$, Lemma~\ref{lem:conditional-ex} implies that 
\[
	\Ex[\Delta_k(\xi_w)^4] = \Ex[\xi_w^{(-)}{}^4] \Ex[(a_k^{\alpha_k} b_k^{2\gamma_k} - \Ex[a_k^{\alpha_k} b_k^{2\gamma_k}] )^4] \Ex[\xi_w^{(+)}]^4.
\]
Note that for $r \le 4m$,
\[
	\Ex[|a_i|^r] = \Ex[|\at_i|^r](\frac{2}{n\beta})^{r/2} \le \frac{2^{r/2} M}{(n\beta)^{ r/2}}.
\]
Then using the inequality $\Ex[(X+Y)^4] \le 8(\Ex[X^4] + \Ex[Y^4])$, we estimate $\Ex[\Delta_k(\xi_w)^4] $ roughly as follows 
\[
	\Ex[\Delta_k(\xi_w)^4]  \le \frac{M'}{(n\beta)^{2\sum\alpha}},
\]
for some constant $M' = M'(m, \varepsilon)$. Thus, the case where $\sum \alpha > 0$ becomes trivial. When $\sum\alpha = 0$, using Lemma~\ref{lem:gamma-distribution}, we also arrive at the desired result.
\end{proof}

\begin{theorem}
Let $p$ be a polynomial. Then as $n\to \infty$,
	\[
		n\sqrt{\beta} (\bra{L_n, p} - \Ex[\bra{L_n, p} ]) \dto \begin{cases}
		 \Normal(0, \sigma_p^2), &\text{if } n\beta \to \infty,\\
		  \Normal(0, \sigma_{p,\alpha}^2), &\text{if } n\beta \to 2\alpha \in (0, \infty),
		\end{cases}
	\]
where $\sigma_p^2 \ge 0$ and $\sigma_{p, \alpha}^2 \ge 0$ are the constants in Corollary~\rm\ref{cor:variance}.
\end{theorem}

By the explicit joint density of Gaussian beta ensembles, one can deduce the following Poincar\'e inequality \cite{Duy2016}
\begin{equation}\label{GE}
	n^2 \beta \Var[\langle L_n, f\rangle] \le 2 \Ex[ \langle L_n, (f')^2\rangle],
\end{equation}
for any differential function $f$ with continuous derivative.
The above inequality, together with the central limit theorem for polynomial test functions, implies a central limit theorem for differentiable functions whose derivative is continuous of polynomial growth (see Theorem~4.9 in \cite{Duy-Nakano-2016}). Moreover, in the regime that $n\beta \to \infty$, since the limit variance does not depend on $\beta$, its explicit formula was known \cite[Theorem~3.5]{Su-book-2015}. To summarize, we get the following result.
\begin{theorem}
	Let $f$ be a differentiable function whose derivative is continuous of polynomial growth. Then as $n\to \infty$ with $n\beta \to \infty$, 
	\[
		\sqrt{\beta}\Big( \sum_{j = 1}^n f(\lambda_j) - \Ex \Big[ \sum_{j = 1}^n (f(\lambda_j)\Big]  \Big)\to \Normal(0, \sigma_f^2),
	\]
where 
\[
		\sigma_f^2 = \frac{1}{2\pi^2} \int_{-2}^2 \int_{-2}^2 \left( \frac{f(x) -f(y)}{x - y}\right)^2 \frac{4 - xy}{\sqrt{4 - x^2}\sqrt{4 - y^2}} dx dy. 
	\]
\end{theorem}

\section{Concluding remark}

We have established a central limit theorem for the linear statistics of differentiable functions. For non-differentiable functions, we only have a few results. Let us mention here two examples. 

\begin{itemize}
\item[(i)] (log-determinant) 
The same argument as in \cite{Tao-Vu-2012} yields a central limit theorem for log-determinants of Gaussian beta ensembles, that is, for fixed $\beta > 0$, as $n \to \infty$,  
\[
	\frac{\log |\det(G\beta E_n)| -\frac12 \log n! + \frac14 \log n}{\sqrt{\frac1\beta\log n}}\dto \Normal(0,1).
\]
In case of GOE and GUE $(\beta = 1,2)$, it is also known that the absolute value of the determinant is equal in distribution to a product of independent random variables, from which the above central limit theorem can be established in a different way. See \cite{Duy-RIMS} for more details.
	
\item[(ii)](eigenvalues counting) Let $f$ be the indicator of the interval $[a,b] \subset (-2,2)$. Then its linear statistics counts the number of eigenvalues falling in this interval. For the case of GUE and GOE, we have 
	\[
		\frac{\#\{\lambda_j \in [a,b]\} - \Ex[\#\{\lambda_j \in [a,b]\}]}{\sqrt{\frac{\log n}{\beta\pi^2}}} \to \Normal(0,1),
	\]
where $\beta = 1,2$ corresponds to GOE and GUE, respectively. The proof for the GUE case is based on estimating the variance and a general central limit theorem for counting points in some region of determinantal point processes. The case GOE is deduced from GUE by interlacing formulae \cite{Dallaporta-Vu-2011, Gustavsson-2005}. However, the same problem for general $\beta$ is still open. To the best of the author's knowledge, only a result for counting of positive eigenvalues was known \cite[Theorem~3.12]{Su-book-2015}. 
\end{itemize}

\bigskip
\noindent{\bf Acknowledgements. }
The author would like to thank the referee for many useful comments. This work is supported by JSPS KAKENHI Grant Numbers JP16K17616.

\begin{footnotesize}

\end{footnotesize}

%

\begin{thebibliography}{10}

\bibitem{Allez12}
Allez, R., Bouchaud, J.P., Guionnet, A.: Invariant beta ensembles and the
  {G}auss-{W}igner crossover.
\newblock Phys. Rev. Lett. \textbf{109}, 094,102 (2012)

\bibitem{Peche15}
Benaych-Georges, F., P{\'e}ch{\'e}, S.: Poisson statistics for matrix ensembles
  at large temperature.
\newblock J. Stat. Phys. \textbf{161}(3), 633--656 (2015)

\bibitem{Billingsley-PnM}
Billingsley, P.: Probability and measure, third edn.
\newblock Wiley Series in Probability and Mathematical Statistics. John Wiley
  \& Sons, Inc., New York (1995).
\newblock A Wiley-Interscience Publication

\bibitem{Dallaporta-Vu-2011}
Dallaporta, S., Vu, V.: A note on the central limit theorem for the eigen-value
  counting function of {W}igner matrices.
\newblock Electron. Commun. Probab. \textbf{16}, 314--322 (2011)

\bibitem{DE02}
Dumitriu, I., Edelman, A.: Matrix models for beta ensembles.
\newblock J. Math. Phys. \textbf{43}(11), 5830--5847 (2002)

\bibitem{DE06}
Dumitriu, I., Edelman, A.: Global spectrum fluctuations for the
  {$\beta$}-{H}ermite and {$\beta$}-{L}aguerre ensembles via matrix models.
\newblock J. Math. Phys. \textbf{47}(6), 063,302, 36 (2006)

\bibitem{Duy-RIMS}
Duy, T.K.: Distributions of the determinants of {G}aussian beta ensembles.
\newblock RIMS K\^oky\^uroku \textbf{No.2023}, 77--85 (2017)

\bibitem{Duy2016}
Duy, T.K.: On spectral measures of random {J}acobi matrices.
\newblock Osaka J. Math.  (2017).
\newblock (to appear). Available at arXiv:1601.01146

\bibitem{Duy-Nakano-2016}
Duy, T.K., Nakano, F.: Gaussian beta ensembles at high temperature: eigenvalue
  fluctuations and bulk statistics.
\newblock arXiv:1611.09476

\bibitem{DS15}
Duy, T.K., Shirai, T.: The mean spectral measures of random {J}acobi matrices
  related to {G}aussian beta ensembles.
\newblock Electron. Commun. Probab. \textbf{20}, no. 68, 13 (2015)

\bibitem{Gustavsson-2005}
Gustavsson, J.: Gaussian fluctuations of eigenvalues in the {GUE}.
\newblock Ann. Inst. H. Poincar\'e Probab. Statist. \textbf{41}(2), 151--178
  (2005)

\bibitem{Johansson98}
Johansson, K.: On fluctuations of eigenvalues of random {H}ermitian matrices.
\newblock Duke Math. J. \textbf{91}(1), 151--204 (1998)

\bibitem{Popescu-09}
Popescu, I.: General tridiagonal random matrix models, limiting distributions
  and fluctuations.
\newblock Probab. Theory Related Fields \textbf{144}(1-2), 179--220 (2009)

\bibitem{Ramirez-Rider-Virag-2011}
Ram\'\i~rez, J.A., Rider, B., Vir\'ag, B.: Beta ensembles, stochastic {A}iry
  spectrum, and a diffusion.
\newblock J. Amer. Math. Soc. \textbf{24}(4), 919--944 (2011)

\bibitem{Shcherbina-2015}
Shcherbina, M.: On fluctuations of eigenvalues of random band matrices.
\newblock J. Stat. Phys. \textbf{161}(1), 73--90 (2015)

\bibitem{Su-book-2015}
Su, Z.: Random matrices and random partitions, \emph{World Scientific Series on
  Probability Theory and Its Applications}, vol.~1.
\newblock World Scientific Publishing Co. Pte. Ltd., Hackensack, NJ (2015)

\bibitem{Tao-Vu-2012}
Tao, T., Vu, V.: A central limit theorem for the determinant of a {W}igner
  matrix.
\newblock Adv. Math. \textbf{231}(1), 74--101 (2012)

\bibitem{Valko-Virag-2009}
Valk\'o, B., Vir\'ag, B.: Continuum limits of random matrices and the
  {B}rownian carousel.
\newblock Invent. Math. \textbf{177}(3), 463--508 (2009)

\end{thebibliography}

\end{document}